\documentclass[reqno,10pt,a4paper]{article}
\usepackage{bm}
\usepackage{amsfonts}
\usepackage{amssymb}
\usepackage{amsmath, amssymb, amsthm, enumerate, latexsym, amstext}
\usepackage{amscd}
\usepackage[all,cmtip]{xy}
\usepackage{pstricks,pst-node, pst-tree}
\usepackage{rotating}
\usepackage{graphicx}
\usepackage{amsopn}
\usepackage{listings}
\usepackage{mathtools}
\usepackage{commath}
\usepackage{tikz-cd}
\usetikzlibrary{cd}
\usetikzlibrary{decorations.pathmorphing}

\usepackage{epstopdf}
\newtheorem{theorem}{Theorem}
\newtheorem{lemma}{Lemma}
\newtheorem{proposition}{Proposition}

\newtheorem{definition}{Definition}

\DeclareMathOperator{\Macp}{MacP}

\title{A counterexample to Las Vergnas' strong map conjecture on realizable oriented matroids}
\author{Pei Wu\\ email: {wupei5060309031@gmail.com} }

\begin{document}
	\maketitle
	
	\begin{abstract}
		The Las Vergnas' strong map conjecture, asserts that any strong map of oriented matroids $f:\mathcal{M}_1\rightarrow\mathcal{M}_2$ can be factored into extensions and contractions. This conjecture is known to be false due to a construction by Richter-Gebert, he finds a strong map which is not factorizable, however in his example $\mathcal{M}_1$ is not realizable. The problem that whether there exists a non-factorizable strong map between realizable oriented matroids still remains open. In this paper we provide a counterexample to the strong map conjecture on realizable oriented matroids, which is a strong map $f:\mathcal{M}_1\rightarrow\mathcal{M}_2$, $\mathcal{M}_1$ is an alternating oriented matroid of rank $4$ and $f$ has corank $2$. We prove it is not factorizable by showing that there is no uniform oriented matroid $\mathcal{M}^{\prime}$ of rank $3$ such that $\mathcal{M}_1\rightarrow\mathcal{M}^{\prime}\rightarrow\mathcal{M}_2$.
		%\keywords{Las Vergnas' strong map conjecture \and Strong map \and Oriented matroid \and Integral linear programming}
		%\PACS{PACS code1 \and PACS code2 \and more}
		%\subclass{52C40 \and 52B40 \and 55R40 }
	\end{abstract}
	
	\section{Background}
	The strong map conjecture, firstly posed by Las Vergnas\cite{allys1991minors}, asserts that any strong map of oriented matroids $f:\mathcal{M}_1\rightarrow\mathcal{M}_2$ can be factored into extensions and contractions. It is known that the conjecture holds for ordinary matroids \cite{higgs1968strong}. And for oriented matroids a counterexample has been constructed by Richter-Gebert \cite{Richter-Gebert1993}. However, $\mathcal{M}_1$ is not realizable in Richter-Gebert's construction. The problem that whether Las Vergnas' conjecture holds when $\mathcal{M}_1$ is realizable still remains open. In this paper, we will present an counterexample disproving this conjecture.
	
	\begin{theorem}
		There is a strong map $f:\mathcal{M}_1\rightarrow\mathcal{M}_2$ with $\mathcal{M}_1$ being an realizable oriented matroid of rank $4$ on $8$ elements and $f$ corank $2$, which is not factorizable into extensions and contractions. 
	\end{theorem}
	
	Las Vergnas' conjecture on realizable oriented matroids has its own significance as a part of the "combinatorial Grassmannian" program \cite{Mnev1993}. The program is stimulated by pioneering works of Gelfand and MacPherson \cite{Gelfand1992,macpherson1991combinatorial}, in \cite{Gelfand1992} they proposed a formula that calculates rational Pontrjagin classes of a differentiable manifold from combinatorial data. In their proof they make use of a modified formulation of Chern-Weil theory. So it is not possible to calculate any $\mathbb{Z}/p\mathbb{Z}$-characteristic classes following same argument. A possible way to remedy this deficit is to adopt the definition of characteristic classes via Grassmannians. Let's recall some standard facts of characteristic classes (see \cite{Bott1982} or \cite{Milnor2004} for a comprehensive treatment). Let $p:E\rightarrow B$ be a real vector bundle on a manifold $B$, a characteristic class of the bundle is an invariant taking value in cohomology ring $H^*(B)$ of certain coefficients. If $p$ is a $\mathbb{R}^k$-bundle, there is a canonical map (up to isotopy) from $B$ to the \emph{infinite real Grassmannian} $G_k^\infty$, coined \emph{Gauss map}, and characteristic classes are pull-backs of certain cohomology classes on infinite real Grassmannian. Such definition is purely topological, so one would expect that we are able to rewrite this definition using combinatorial data with less effort. MacPherson\cite{Mnev1993} suggests the following object as a substitute of $G_k(\mathbb{R}^n)$: the (chain complex of) poset of all oriented matroids of rank $k$ on $n$ elements, ordering with respect to weak maps, called \emph{MacPhersonian} and denoted as $\Macp(n,k)$. Let $\mathcal{F}^n$ be the free oriented matroid of rank $n$, $\Macp(n,k)$ is the poset of rank $k$ strong image of $\mathcal{F}^n$. One can obtain more general object by substituting $\mathcal{F}^n$ with an arbitrary rank $n$ oriented matroid $\mathcal{M}$ (one can assume $\mathcal{M}$ is realizable for our purpose), called \emph{OM-Grassmannian} and denoted as $\mathcal{G}_k(\mathcal{M})$. Combinatorial Grassmannian program is the study of homopoty type of $\mathcal{G}_k(\mathcal{M})$. The conjecture that $\mathcal{G}_k(\mathcal{M})$ and $G_k(\mathbb{R}^n)$ are homotopy equivalent has been disapproved by Gaku Liu \cite{Liu2016}. And whether $\Macp(n,k)$ and $G_k(\mathbb{R}^n)$ are homotopy equivalent still remains open.
	
	The Las Vergnas' strong map conjecture is related with combinatorial Grassmannian program in the following way: the (non-compact) \emph{Stiefel manifold} $V_k(\mathbb{R}^n)$ is the set of all $k$-tuples of linearly independent vectors, there is a surjective mapping $p:V_k(\mathbb{R}^n)\rightarrow G_k(\mathbb{R}^n)$ by sending the $k$-tuples to the linear space they span. For every $pt\in G_k(\mathbb{R}^n)$, $p^{-1}(\{pt\})$ is isomorphic to $GL(k,\mathbb{R})$, so $V_k(\mathbb{R}^n)$ can be viewed as a principal $GL(k,\mathbb{R})$-bundle over $G_k(\mathbb{R}^n)$. The oriented matroid counterpart of Stiefel manifold is defined as follows: let $\mathcal{M}$ be an oriented matroid of rank $n$, the \emph{OM-Stiefel space} $\mathcal{V}_k(\mathcal{M})$ is defined as all "non-degenerate" $n-k$ extensions, i.e. if the set of new elements is $S$, the contraction $\mathcal{M}/S$ should has rank $k$. So there is a poset mapping $\tilde{p}:\mathcal{V}_k(\mathcal{M})\rightarrow\mathcal{G}_k(\mathcal{M})$ defined by contracting $S$. A natural problem is, is preimage of every point is homotopic to $GL(k,\mathbb{R})$? Note that the Las Vergnas' strong map conjecture would implies the subjectivity of $\tilde{p}$. Our counterexample indicates that there is a point with empty preimage.
	
	\section{Oriented Matroids}
	For completeness we will include a brief introduction to the theory of oriented matroids, in which we try to cover most conventions and facts we use in this paper, one could refer to \cite{bjorner2000oriented} for a detailed treatment. 
	
	Datum of oriented matroid can be encoded by \emph{circuits, vectors, cocircuits, covectors, topes} or \emph{chirotope}. Let $E$ be the ground set. Circuits, vectors, cocircuits, covectors, topes are all signed vectors on $E$. A \emph{signed vector} $X$ is a mapping $X:E\to\{-1,0,1\}$. $X^{-1}(1)$ and $X^{-1}(-1)$ are denoted as $X^+$ and $X^-$, respectively. We will use two ways to write the signed vectors, for example when $E=\{1,\dots,5\}$, $X^{+}=\{1,3\}$ and $X^{-}=\{2,4\}$, $X=(+-+-0)$ or $X=1\bar23\bar4$. $\bm{0}$ is the signed vector $X$ with $X^+=X^-=\emptyset$, $\bm{1}$ is the signed vector $X$ with $X^+=X$. If $X$ is a signed vector, define $-X$ to be the signed vector with $(-X){(i)}=-X{(i)}$, which is called the \emph{opposite} of $X$. Given a set of signed vectors $\mathcal{X}$, \emph{reorientation} of an element $e\in E$ is the operation reversing values of $X(e)$ for all signed vectors $X\in\mathcal{X}$. The \emph{support} of a signed vector is defined as $X^+\cup X^-$, denoted as $\underline{X}$, the \emph{size} of $X$ is defined as the size of support, signed vector $X$ has \emph{full support} iff $\underline{X}=E$. Two signed vectors $Y$ and $Z$ are \emph{perpendicular} iff in their component-wise products $X$, $X^+$ and $X^-$ are all empty or all non-empty, written as $Y\perp Z$. There is a natural partial ordering on signed vectors: $X\preceq X^{\prime}$ iff $X^+\subseteq X^{\prime+}$ and $X^-\subseteq X^{\prime-}$. If $E^{\prime}\subseteq E$, \emph{restriction} of $X$ on $E^{\prime}$ is a signed vector on $E^{\prime}$, defined as $X|_{E^{\prime}}(i)=X(i)$ for $i\in E^{\prime}$. The \emph{chirotope} is an anti-symmetric mapping $\chi:E^r\rightarrow\{1,0,-1\}$, in which $r=r(\mathcal{M})$ is the \emph{rank} of the oriented matroid. An oriented matroid can be encoded by a set of circuits, or cocircuits, etc, satisfying certain sets of axioms (\cite{bjorner2000oriented} Chapter. 3). For completeness, we include the covector axiomatization of oriented matroids here:
	
	\begin{definition}
		An {oriented matroid} is a pair $\mathcal{M}=(E,\mathcal{L})$, covectors $\mathcal{L}$ is a set of signed vector on $E$ such that:
		\begin{enumerate}
			\item $\bm{0}\in \mathcal{L}$
			\item $X\in \mathcal{L}\implies -X\in \mathcal{L}$
			\item $X,Y\in \mathcal{L}\implies X\circ Y\in \mathcal{L}$
			\item (covector elimination) $X,Y\in \mathcal{L},\, e\in S(X,Y)\implies$ there exist $Z\in \mathcal{L}$ such that $Z{(e)}=0$ and $Z{(f)}=(X\circ Y){(f)}$ for $f\not\in S(X,Y)$.
		\end{enumerate} 
		
		In which $S(X,Y):=\{e\in E|X{(e)}=-Y{(e)}\not=0\}$ and $X\circ Y$ is the signed vector defined as 
		\[(X\circ Y){(e)}=\begin{cases}
		X{(e)}, & \text{if }X{(e)}\not=0 \\
		Y{(e)}, & \text{otherwise}
		\end{cases}\]
	\end{definition} 
	
	A finite set of points $E=\{\bf{v}_1,\dots,\bf{v}_n\}$ in affine space $\mathbb{R}^{r-1}$ is a \emph{point configuration} if their affine closure is $\mathbb{R}^{r-1}$, we can associate it with an oriented matroid $\mathcal{M}$. Each affine dependency $\sum \lambda_i {\bf{v}}_i={\bf 0}, \sum \lambda_i=0$ defines a vector $X$ of $\mathcal{M}$ by $X^+=\{{\bf{v}}_i|\lambda_i>0\}$ and $X^-=\{{\bf{v}}_i|\lambda_i<0\}$. Geometrically this implies the convex hull of $X^+$ and $X^-$ are intersecting at interior points.
	And each ${\bf{w}}\in \mathbb{R}^{r*},\, a\in \mathbb{R}$ defines a covector $X$ of $\mathcal{M}$ such that $X^+=\{{\bf{v}_i|\left\langle\bf{v}_i,\bf{w}\right\rangle}>a\}$, $X^-=\{{\bf{v}_i|\left\langle\bf{v}_i,\bf{w}\right\rangle}<a\}$. That is, $X^+$ and $X^-$ lie in two half-spaces cut by hyperplane $\{{\bf{v}_i|\left\langle\bf{v}_i,\bf{w}\right\rangle}=a\}$.  Circuits are non-zero $\preceq$-minimal vectors and cocircuits are non-zero $\preceq$-minimal covectors and topes are $\preceq$-maximal covectors. The chirotope is an alternating function on $E^r$, $\chi:E^r\to \{-1,0,1\}$, defined by $\chi(i_1,\dots,i_r)=\text{sign}(\det(\bf{v}_{i_1}-\bf{v}_{i_{r}},\dots,\bf{v}_{i_{r-1}}-\bf{v}_{i_{r}}))$ ($-\chi$ is considered to be same chirotope as $\chi$ ). An oriented matroid is \emph{realizable} iff it arises in this way for some $\{{\bf v}_e:e\in E\}$, up to reorientation of elements. One could verify that every vector is perpendicular to every covector, which is a property also holds for non-realizable oriented matroids. 
	
	An oriented matroid is acyclic iff $\bm{1}$ is a covector. An oriented matroids is \emph{uniform} iff $\chi(i_1,\dots,i_r)\not=0$ for all $i_1,\dots,i_r$ distinct. In an uniform oriented matroid the size of circuits are always $r+1$ and size of cocircuits are always $n-r$. Define $\Phi_r(n):=\sum_{i=0}^r\binom{n}{i}$, the number of topes is $2\Phi_{r-1}(n-1)$. Actually the converse is also true by \cite{Gartner1994}, which provides an alternative axiomatization of uniform oriented matroid, for which will be useful for enumerating oriented matroids. 
	
	\begin{theorem}\label{vc}
		Let $\mathcal{T}$ be a set of full support signed vectors on $[n]$, $\mathcal{T}$ is the set of topes of an rank $r$ uniform oriented matroid iff:
		\begin{enumerate}
			
			\item $\#T=2\Phi_{r-1}(n-1)$
			\item $X\in \mathcal{L}\implies -X\in \mathcal{L}$
			\item (VC-dimension) For any $Q\in\binom{[n]}{r+1}$, there exist a signed vector $c_Q$(together with its opposite) supported on $Q$ such that $T\perp c_Q$(or equivalently, $T|_Q\not= c_Q$) for every $T\in\mathcal{T}$. 
			
		\end{enumerate} 
	\end{theorem}
	
	We further define several operations on oriented matroid for stating the \emph{strong-map conjecture} of Las Vergnas. If $\mathcal{M}_1$ is on ground set $E_1$ and $\mathcal{M}_2$ is on ground set $E_1\cup\{u\}$, they were of the same rank and their chirotope coincide on $E_1$ we say $\mathcal{M}_2$ is a \emph{single extension} of $\mathcal{M}_1$, and $\mathcal{M}_1$ is a \emph{single deletion} of $\mathcal{M}_2$ by deleting $u$, written as $\mathcal{M}_1=\mathcal{M}_2\backslash u$, or $\mathcal{M}_2\hookrightarrow \mathcal{M}_1$. An extension is a  composite of single extensions and a deletion is a composite of single deletions. If $E^{\prime}\subseteq E_1$, the \emph{restriction} of $\mathcal{M}_1$ on $E^{\prime}$ is the oriented matroid that deletes all elements not in $E^{\prime}$. \emph{Contractions} are defined as follows: if $\mathcal{M}_1$ is on ground set $E$ and $u\in E$, contraction of $u$ is defined as a oriented matroid $\mathcal{M}_2$ on $E\setminus\{u\}$ with chirotope $\chi_2(x_1,\dots,x_{r-1})=\chi_1(u,x_1,\dots,x_{r-1})$, written as $\mathcal{M}_1=\mathcal{M}_2/ u$ or $\mathcal{M}_1\twoheadrightarrow\mathcal{M}_2$. The contraction of a subset $U\subseteq E$ is composite of contracting all elements in $U$.
	
	Define there is a \emph{strong map} from $\mathcal{M}_1$ to $\mathcal{M}_2$ iff they are on same ground set and every covector of $\mathcal{M}_1$ is a covector of $\mathcal{M}_2$, in this case, we write the strong map $f:\mathcal{M}_1\rightarrow\mathcal{M}_2$ (for a general discussion see  \cite{bjorner2000oriented} \textit{pp.} 319), the corank of a strong map is defined as $r(\mathcal{M}_1)-r(\mathcal{M}_2)$. A composition of extensions and contractions on a same set of elements is always a strong map (we say such strong map \emph{factorizable} for short), the strong map conjecture asks whether the converse is true. It is known that the conjecture holds if corank is $1$, rank of $\mathcal{M}_2$ is $1$ or rank of $\mathcal{M}_1$ only one less than the size of ground set(\cite{richter1994zonotopal}, Exercise 7.30 in \cite{bjorner2000oriented}). The following proposition gives a equivalent condition for strong maps on uniform oriented matroids.
	\begin{proposition}\label{1}
		Let $\mathcal{T}_1,\,\mathcal{T}_2$ be topes of oriented matroids $\mathcal{M}_1,\,\mathcal{M}_2$, respectively. We further assume $\mathcal{M}_2$ is uniform, then there exists a strong map $\mathcal{M}_1\rightarrow\mathcal{M}_2$ iff $\mathcal{T}_2\subseteq\mathcal{T}_1$.
	\end{proposition}
	
	\begin{proof}
		The "only if" part is trivial since every tope is a covector, we will prove the "if" part. Let $\mathcal{C}_1,\,\mathcal{C}_2$ be covectors of oriented matroids $\mathcal{M}_1,\,\mathcal{M}_2$, respectively. Suppose $X\in \mathcal{C}_2$, then since $\mathcal{M}_2$ is uniform, any $X^\prime\succeq X$ is a covector of $\mathcal{M}_2$. Let $\mathcal{P}(X)=\{X^\prime:X^\prime\succeq X,\, X^\prime\text{ has full support}\}$, we have $\mathcal{P}(X)\subseteq\mathcal{T}_2\subseteq\mathcal{T}_1\subseteq\mathcal{C}_1$.
		
		Then we prove $X$ is a covector of $\mathcal{M}_1$. Observe that if $X_1$, $X_2$ are two covectors that only differ in one index (i.e. there exists $i\in E$ s.t. $X_1(e)=X_2(e)$ for $e\not=i$ and $X_1(i)=-X_2(i)$), then by covector elimination $X^\prime$ with $X^\prime(e)=X_2(e)$ for $e\not=i$ and $X^\prime(i)=0$ is a covector, applying this property on $\mathcal{P}(X)$ iteratively we have $X\in \mathcal{C}_1$. Thus $\mathcal{M}_1\rightarrow\mathcal{M}_2$.
	\end{proof}
	\subsection{Alternating oriented matroid}
	Alternating oriented matroid is an important family of oriented matroids with many nice properties. An alternating oriented matroid is an oriented matroid on $[n]$ with rank $r$, with chirotope: $\chi(e_1,\dots,e_r)=1$ if $1\le e_1\le\dots\le e_r\le n$. The main fact we need is:  topes of a rank $k$ alternating oriented matroid are all signed vectors with at most $k-1$ sign changes. For example, if $k=4$, topes are signed vectors with the form $(+\dots+-\dots-)$, $(+\dots+-\dots-+\dots+)$ or $(+\dots+-\dots-+\dots+-\dots-)$ or their opposite. Alternating oriented matroids are always realizable by momentum curve $t\mapsto (t,\dots,t^{r-1}),\,t\in [n]$.
	\section{Construction and Verification of Counterexample}
	For simplicity we will always consider those signed vectors with first non-zero component positive from now on, because we can identify oppositely signed vectors. The counterexample is following strong map $f:\mathcal{M}_1\rightarrow\mathcal{M}_2$. $\mathcal{M}_1$ is a rank $4$ alternating oriented matroid on ground set $E=[n]:=\{1,\dots,n\}$ with $n$ even. $\mathcal{M}_2$ is a rank $2$ oriented matroid defined as follows: let $\sigma$ be the permutation $(1\, 2)(3\, 4)\dots(n-1\,n)$, chirotope of $\mathcal{M}_2$ is defined as $\chi(i,j)=1$ iff $\sigma(i)\ge\sigma(j)$. Topes of $\mathcal{M}_2$ were all in forms of $(+\dots+-\dots-)$ or $(+\dots+-+-\dots-)$. Thus by Proposition \ref{1}, $f:\mathcal{M}_1\rightarrow\mathcal{M}_2$ is a strong map indeed.
	
	We first give an intuitive (and invalid) explanation of why $f$ is not factorizable when $n$ is big enough. Note that $\mathcal{M}_1$ can be realized by moment curve $t\in [n],\,h:t\mapsto(t,t^2,t^3)$, we could extend it for $t\in \mathbb{R}$. And if $f$ is a factorizable strong map, it can be realized as a linear projective transformation, let the transformation be $g$, then $g(h(t))$ is a rational function in form of $p_1(t)/p_2(t)$, in which $p_1,p_2$ are polynomials at most cubic. So the number of $t$ with $(g\circ h)^{\prime}(t)=0$ is at most $4$ and the number of poles is at most $3$. So we could realize $\mathcal{M}_2$ as $[n]\rightarrow \mathbb{R}$, which is the restriction of $g\circ h$ on $[n]$, then for every $i=2,\dots,n-1$, either $g(h(i))<g(h(i+1)),\,g(h(i))<g(h(i-1))$ or $g(h(i))>g(h(i+1)),\,g(h(i))>g(h(i-1))$ holds, which means there is a critical point or pole of $g\circ h$ near $i$, but the number of points satisfying such condition is at least $n-2$, which leads to a contradiction.
	
	This argument is not valid due to two reasons. Firstly the realization does not necessarily be the moment curve, secondly the extension may be not realizable. We will give a strict proof that when $n=8$, $f$ is not factorizable. 
	
	Observe that if $f$ is a factorizable strong map, then since deletions and contractions commute there exists an oriented matroid $\mathcal{M}^{\prime}$ of rank $3$ such that $f_1\circ f_2$ where $f_1:\mathcal{M}_1\rightarrow\mathcal{M}^{\prime}$, $f_2:\mathcal{M}^{\prime}\rightarrow\mathcal{M}_2$ are both strong maps. We can further assume that $\mathcal{M}^{\prime}$ is uniform by perturbing the extension element (Proposition 7.2.2(2) in \cite{bjorner2000oriented}).
	
	To show that $\mathcal{M}^{\prime}$ do not exist, we start from considering the case $n=6$:
	
	\begin{lemma} \label{6}
		Let $\mathcal{M}_1,\,\mathcal{M}_2$ be two oriented matroids defined above for the case of $n=6$. If $\mathcal{M}^{\prime}$ is an uniform oriented matroid of rank $3$ such that $f_1:\mathcal{M}_1\rightarrow\mathcal{M}^{\prime}$, $f_2:\mathcal{M}^{\prime}\rightarrow\mathcal{M}_2$ are both strong maps. Then $(+-+-00),\,(+-00-+)$ are circuits of $\mathcal{M}^{\prime}$. 
	\end{lemma}
	
	The can be done by a brute force search of possible $\mathcal{T}(\mathcal{M}^{\prime})$, which should satisfy $\#\mathcal{T}(\mathcal{M}^{\prime})=16$ and $\mathcal{T}(\mathcal{M}_2)\subset\mathcal{T}(\mathcal{M}^{\prime})\subset\mathcal{T}(\mathcal{M}_1)$, note that $\#\mathcal{T}(\mathcal{M}_1)=26$ and $\#\mathcal{T}(\mathcal{M}_2)=6$, so there are $\binom{26-6}{16-6}=184,756$ cases to check. And another constraint is the VC-dimensional property defined in Theorem. \ref{vc}: for every $Q\in \binom{[6]}{4}$, there exist a signed vector $c_Q$ supported on $Q$ such that $c_Q\perp T$ for every $T\in \mathcal{T}(\mathcal{M}^{\prime})$(See \textit{Supplementary File}, or \texttt{https://github.com/PeterWu-Biomath/OM-Stong-Map} for an implementation of the proposed algorithm and explanation of code). There are $20$ sets of signed vectors satisfy this condition. For all those possibilities, $(+-+--),(+----+)\not \in \mathcal{T}(\mathcal{M}^{\prime})$ always holds, note that $(+-+--)$ is the only signed vector in $\mathcal{T}(\mathcal{M}_1)$ perpendicular to $(+-+-00)$, hence $\forall T\in \mathcal{T}(\mathcal{M}^{\prime})$, $T\perp (+-+-00)$. Thus $(+-+-00)$ is a circuit of $\mathcal{M}^{\prime}$. Following same argument, $(+-00-+)$ is a circuit of $\mathcal{M}^{\prime}$.
	
	Finally, the nonfactorzability of case $n=8$ follows immediately from Lemma \ref{6}. Restricting on $\{1,2,3,4,5,6\}$ we know the circuit of $\mathcal{M}^{\prime}$ supported on $\{1,2,5,6\}$ is $(+-00-+00)$, however by restricting on $\{1,2,5,6,7,8\}$ the circuit on $\{1,2,5,6\}$ should be $(+-00+-00)$, which leads to a contradiction.
	
	\section*{acknowledgements}
		I'm very grateful to Andreas Dress, David Bryant, Jack Koolen, Luis Goddyn, Paul Tupper and Stefan Gr\"unewald for useful remarks. And many thanks go in particular to Luis Goddyn for his suggestions on manuscripts.

	\bibliographystyle{spmpsci}      % mathematics and physical sciences
	%\bibliography{strong_map}   % name your BibTeX data base

\end{document}